\documentclass[10pt, letterpaper]{article}

\usepackage[T1]{fontenc}
\usepackage[utf8]{inputenc}
\usepackage{hyperref}
\hypersetup{
    colorlinks,
    allcolors=blue
}

\usepackage{enumerate}

\usepackage{amssymb}
\usepackage{amsmath}
\usepackage{amsthm}
\usepackage{mathtools}
    \mathtoolsset{centercolon}

\usepackage{tikz}
\usetikzlibrary{cd}
\usetikzlibrary{calc}

\usepackage{float}

\usepackage{biblatex}
\newtheoremstyle{mytheoremstyle} %
    {\topsep}                    %
    {\topsep}                    %
    {}                           %
    {}                           %
    {\bf}                           %
    {.}                          %
    {.5em}                       %
    {}  %

\newtheoremstyle{mywoproofstyle} %
    {\topsep}                    %
    {1.3\topsep}                    %
    {}                           %
    {}                           %
    {\bf}                           %
    {.}                          %
    {.5em}                       %
    {}  %

\theoremstyle{mywoproofstyle}
\theoremstyle{mytheoremstyle}

\newtheorem{corollary}[subsection]{Corollary}
\newtheorem{definition}[subsection]{Definition}
\newtheorem{example}[subsection]{Example}
\newtheorem{lemma}[subsection]{Lemma}

\newtheorem{theorem}[subsection]{Theorem}

\renewcommand{\phi}{\varphi}

\DeclarePairedDelimiter\abs{\lvert}{\rvert}%
\DeclarePairedDelimiter\gen{\langle}{\rangle}%
\DeclarePairedDelimiterX\cgen[2]{\langle}{\rangle}%
 {#1\;\delimsize\vert\;#2}

\DeclarePairedDelimiter\set{\{}{\}}%
\DeclarePairedDelimiterX\cset[2]{\{}{\}}%
 {#1\,\delimsize\vert\,#2}

\newcommand{\myhat}[1]    {\ensuremath{\widehat{#1}}}

\DeclareMathOperator{\Aut}{Aut}
\DeclareMathOperator{\id}{id}

\DeclareMathOperator{\Out}{Out}
\DeclareMathOperator{\Soc}{Soc}
\DeclareMathOperator{\Sym}{Sym}

\begin{document}
\title{Towards Efficient Normalizers of Primitive Groups}
\author{Sergio Siccha\textsuperscript{1}}%
\date{
\textsuperscript{1}%
\href{https://orcid.org/0000-0002-2839-5265}{[0000-0002-2839-5265]}
}
\maketitle              %
\begin{abstract}
\noindent
We present the ideas behind an algorithm to compute normalizers of primitive
groups with non-regular socle in polynomial time.
We highlight a concept we developed called permutation morphisms
and present timings for a partial implementation of our algorithm.
This article is a collection of results from the author's PhD thesis.
\\[8pt]
\end{abstract}

\section{Introduction}
\label{sec:intro}

One of the tools to study the internal structure of groups
is the normalizer.
For two groups $G$ and $H$, which are contained in a common overgroup $K$, we
call the \emph{normalizer of $G$ in $H$},
denoted $N_H(G)$,
the subgroup of $H$ consisting of those elements that leave $G$ invariant under
conjugation.

We only consider finite sets, finite groups, and permutation groups acting on
finite sets.
We assume permutation groups to always be given by generating sets and say that
a problem for permutation groups can be solved in \emph{polynomial time},
if there exists an algorithm which, given permutation groups of degree $n$,
solves it in time bounded polynomially in $n$ and in the sizes of the given
generating sets.
While many problems for permutation groups can be solved efficiently both in
theory and in practice,
no polynomial time algorithm to compute normalizers of permutation groups is
known.

A transitive permutation group $G$ acting on a set $\Omega$ is called
\emph{primitive} if there exists no non-trivial $G$-invariant partition of
$\Omega$.
Primitive groups have a rich and well-understood structure.
Hence many algorithms use the natural recursion from general permutation groups
to transitive and in turn to primitive ones.
For two permutation groups $G, H \leq \Sym \Omega$ computing the normalizer of
$G$ in $H$ in general is done by searching for the normalizer of $G$ in the
symmetric group $\Sym \Omega$ and simultaneously computing the intersection
with $H$.
We focus on computing the normalizer of a primitive group $G \leq \Sym
\Omega$ in $\Sym \Omega$.
Being able to compute normalizers for primitive groups efficiently may lead to
improved algorithms for more general situations.

Our results build substantially on the O'Nan-Scott classification of primitive
groups, see \cite{liebeck-praeger-saxl:onan-scott}, and on the classification
of finite simple groups (CFSG).

Recall that the \emph{socle} of a group $G$, denoted $\Soc G$, is the subgroup
generated by all minimal normal subgroups of $G$.
Our theoretical main result is the following theorem.
\begin{theorem}[{\cite[Theorem~9.1]{siccha:phd-thesis}}]
\label{thm:normsym-primitive-nonab-soc-poly}
    Let a primitive group $G \leq \Sym \Omega$ with non-regular socle%
\footnote{This excludes groups of affine and of twisted wreath type.}
    be given.
    Then we can compute $N_{\Sym \Omega}(G)$ in polynomial time.
\end{theorem}

As is often the case in computational group theory, ideas from theoretical
algorithms can be employed in practical algorithms and vice versa.
While the algorithms in \cite{siccha:phd-thesis} are primarily
theoretical ones,
we also provide probabilistic nearly-linear time versions where possible.
The author is developing the GAP package
\texttt{NormalizersOfPrimitiveGroups},
hosted at
\begin{center}
\url{https://github.com/ssiccha/NormalizersOfPrimitiveGroups}%
\footnote{may move to %
\url{https://github.com/gap-packages/NormalizersOfPrimitiveGroups}},
\end{center}
with the aim to implement practical versions of the algorithms developed in
\cite{siccha:phd-thesis}.
Until now, algorithms concerning permutation morphisms and primitive
groups of type PA are implemented.
First experiments indicate that already for moderate degrees these outperform
the GAP built-in algorithm \texttt{Normalizer} by several orders of magnitude,
see Table~\ref{tab:runtimes}.

Since no polynomial time solutions are known for the normalizer problem, the
generic practical algorithms resolve to backtracking over the involved groups
in one way or another.
The fundamental framework of modern backtrack algorithms for permutation groups
is Leon's partition backtrack algorithm~\cite{leon},
which generalizes previous backtrack approaches
\cite{butler:backtrack,butler:normalizers,holt:normalizers,sims:backtrack}
and generalizes ideas of \emph{nauty} \cite{mckay-piperno:nauty-traces} to
the permutation group setting.
Partition backtrack is implemented in GAP
\cite{gap}
and Magma
\cite{magma}.
Recently, the partition backtrack approach was generalized
to a ``graph backtrack'' framework
\cite{JPWW}.

Theißen developed a normalizer algorithm which uses orbital graphs to prune the
backtrack search \cite{theissen}.
Chang is currently developing specialized algorithms for highly
intransitive permutation groups, her PhD thesis should appear shortly.
It is to expect that the work in \cite{JPWW} can also be extended to normalizer
problems.
Hulpke also implemented normalizer algorithms in
\cite{hulpke:normalizers} using group automorphisms
and the GAP function \texttt{NormalizerViaRadical} based on
\cite{glasby-slattery:normalizers}

In Section~\ref{sec:strat} we outline the strategy behind our algorithms.
In Section~\ref{sec:onan} we recall the O'Nan-Scott Theorem.
We present our new concept of permutation morphisms in
Section~\ref{sec:permmors}.
In Section~\ref{sec:red} we sketch how we use our results
to obtain Theorem~\ref{thm:normsym-primitive-nonab-soc-poly}.
In Section~\ref{sec:implementation} we discuss our implementation.

\section{Strategy}
\label{sec:strat}
We describe the strategy of the theoretical algorithm behind
Theorem~\ref{thm:normsym-primitive-nonab-soc-poly}.
Comments regarding the implementation of its building blocks are given at the
end of each following section.

In this section let $G \leq \Sym \Omega$ be a primitive group with non-regular
socle $H$.
The normalizer of $H$ in $\Sym \Omega$ plays a central role in our algorithm,
in this section we denote it by $M$.
Observe that to compute $N_{\Sym \Omega}(G)$ it suffices to compute
$N_M(G)$ since the former is contained in $M$.

The socle $H$ is isomorphic to $T ^ \ell$ for some finite non-abelian simple
group $T$ and some positive integer $\ell$.
The group $G$ is isomorphic to a subgroup of the wreath product
$\Aut(T) \wr S_\ell$,
see Section~\ref{sec:onan} for a definition of wreath products.
By the O'Nan-Scott~Theorem
the respective isomorphism extends to an embedding%
\footnote{For twisted wreath type the situation is slightly more complicated.}
of the normalizer $M$ into $\Aut(T) \wr S_\ell$.
Furthermore $\ell$ is of the order $O(\log \abs \Omega)$.
Hence the index of $G$ in $M$, and thus also the search-space of the
normalizer computation $N_M(G)$, is tiny in comparison to the index of $G$ in
$\Sym \Omega$.

Our approach can be divided into two phases.
First we compute $M$,
this is by far the most labor intensive part.
To this end we compute a sufficiently well-behaved conjugate
of $G$, such that we can exhibit the wreath structure mentioned above.
In \cite{siccha:phd-thesis} we make this more precise and define a
\emph{weak canonical form} for primitive groups.
Using that conjugate and the O'Nan-Scott Theorem we can write down generators
for $M$.
In the second phase,
we compute a reduction homomorphism
$\rho : M \to S_k$
with $k \leq 6 \log \abs \Omega$.
After this logarithmic reduction,
we use
Daniel Wiebking's simply exponential time
algorithm
\cite{wiebking:normalizers-simply-exponential,%
wiebking:normalizers-simply-exponential-siam},
which is based on the canonization framework%
~\cite{schweitzer-wiebking:canonisation-framework},
to compute $N_{S_k}(\rho(G))$.
Note that the running time of a simply exponential time algorithm
called on a problem of size $\log n$ is
$2 ^ {O(\log n)}$
and thus is bounded by $2 ^ {c \log n} = n ^ c$ for some constant $c > 0$.
Then we use Babai's famous quasipolynomial time algorithm for graph-isomorphism
\cite{babai:graphiso-quasipolynomial,babai:graphiso-quasipolynomial-acm}
to compute the group intersection
$N_{\rho(M)}(\rho(G)) = \rho(M) \cap N_{S_k}(\rho(G))$.
Notice that since we perform these algorithms on at most $6 \log n$
points they run in time polynomial in $n$.
The homomorphism $\rho$ is constructed in such a way,
that computing the preimage of the above normalizer $N_{\rho(M)}(\rho(G))$
yields $N_M(G)$.
Recall that $N_M(G)$ is equal to $N_{\Sym \Omega}(G)$.

In our implementation we do not use the algorithms by Wiebking and Babai since
these are purely theoretical.
Instead we use the partition backtrack implemented in GAP.

\section{The O'Nan-Scott Theorem}
\label{sec:onan}

The goal of this and the next section is to illustrate
how we use the O'Nan-Scott Theorem to prove the following theorem.
In this article we limit ourselves to groups of type PA, which we define
shortly.
\begin{theorem}
[{\cite[Theorem~8.1]{siccha:phd-thesis}}]
\label{thm:normsoc-poly-primitive-nonab-soc}
    Let a primitive group $G \leq \Sym \Omega$ with non-abelian socle be given.
    Then we can compute $N_{\Sym \Omega}(\Soc G)$ in polynomial time.
\end{theorem}
\begin{proof}
For groups of type PA this will follow from
Corollary~\ref{cor:norm-of-cwise-poly}
and
Lemma~\ref{lem:pd}.
\end{proof}

The O'Nan-Scott Theorem classifies
how the socles of primitive groups can act,
classifies the normalizers of the socles,
and determines criteria to decide which subgroups of these normalizers act
primitively.
We follow the division of primitive groups into eight O'Nan-Scott
types
as it was suggested by L\'aszl\'o G. Kov\'acs
and first defined by Cheryl Praeger in
\cite{praeger:onan-scott-eight-types}.
In this section we define the types AS and PA and recall some of their
basic properties.
In particular we describe the normalizer of the socle for groups of type PA
and how to construct the normalizer of the socle, if the group is given in a
sufficiently well-behaved form.

The version of the O'Nan-Scott Theorem we use,
for a proof see \cite{liebeck-praeger-saxl:onan-scott},
is:
\begin{theorem}
\label{thm:onan-scott}
    Let $G$ be a primitive group on a set $\Omega$.
    Then $G$ is a group of type
    HA,
    AS, PA,
    HS, HC,
    SD, CD,
    or TW.
\end{theorem}

The abbreviation AS stands for {\bf A}lmost {\bf S}imple.
A group is called \emph{almost simple} if it contains a non-abelian simple
group and can be embedded into the automorphism group of said simple group.
A primitive group $G$ is of \emph{AS type} if its socle is a
non-regular non-abelian simple group.

The abbreviation PA stands for {\bf P}roduct {\bf A}ction.
The groups of AS type form the building blocks for the groups of PA
type.
To define this type, we shortly recall the notion of wreath products and their
product action.

The \emph{wreath product} of two permutation groups $H \leq \Sym \Delta$
and $K \leq S_d$
is denoted by $H \wr K$ and defined as the semidirect product
$H ^ d \rtimes K$ where $K$ acts per conjugation on $H ^ d$ by permuting its
components.
We identify $H ^ d$ and $K$ with the corresponding subgroups of $H \wr K$
and call them the \emph{base group} and the \emph{top group}, respectively.

For two permutation groups $H \leq \Sym \Delta$ and
$K \leq S_d$ the \emph{product action} of the wreath product $H \wr K$
on the set of tuples $\Delta ^ d$ is given by letting
the base group act component-wise on $\Delta ^ d$
and letting the top group act by permuting the components of $\Delta ^ d$.

\begin{definition}
\label{def:primitive-PA}
    Let $G \leq \Sym \Omega$ be a primitive group.
    We say that $G$ is of \emph{type PA} if there exist an
    $\ell \geq 2$ and a primitive group $H \leq \Sym \Delta$ of type AS
    such that $G$ is
    permutation isomorphic to
    a group
    $\myhat G \leq \Sym \Delta ^ \ell$
    with
    \[
        \hspace{3em}
        (\Soc H) ^ \ell \leq
        \myhat G
        \leq H \wr S_\ell
    \]
    in product action on $\Delta ^ \ell$.
\end{definition}
The product action wreath products
$A_5 \wr \gen{(1,2,3)}$
and
$A_5 \wr \gen{(1,2)}$
are examples for primitive groups of type PA.

Let $\myhat G \leq \Sym (\Delta ^ \ell)$ and $H \leq \Sym \Delta$ be as in
Definition~\ref{def:primitive-PA}.
We sketch how to construct the normalizer of the socle of $\myhat G$.
Let $T := \Soc H \leq \Sym \Delta$.
Since $\myhat G$ is given acting in product action we can
read off $H$ and thus compute $T$.
By~\cite[Lemma 4.5A]{dixon-mortimer:permutation-groups} we know that
the normalizer of $\Soc \myhat G$ in $\Sym \Delta ^ \ell$ is
$N_{\Sym \Delta}(T) \wr S_\ell$.
By recent work of Luks and Miyazaki we can compute the normalizer of $T$,
in polynomial time
\cite[Corollary~3.24]{luks-miyazaki:normalisers-poly}.
More precisely this approach yields the following corollary:

\begin{corollary}
\label{cor:norm-of-cwise-poly}
    Let $G \leq \Sym( \Delta ^ \ell )$ be a primitive group of type PA
    with socle $T ^ \ell$ in component-wise action on $\Delta ^ \ell$.
    Then $N_{\Sym(\Delta ^ \ell)}(T ^ \ell)$ can be computed in polynomial
    time.
\end{corollary}

In the practical implementation we use the GAP built-in algorithm to compute
the normalizer of $T$ in $\Sym \Delta$.
Our long-term goal is to use the constructive recognition provided by the
\texttt{recog} package~\cite{recog1.3.2}.
Computing the normalizer of $T$ in $\Sym \Delta$ is then only a matter of
iterating through representatives for the outer automorphisms of $T$.

\section{Permutation morphisms}
\label{sec:permmors}

In general a group of PA type might be given on an arbitrary set and needs only
be permutation isomorphic to a group in product action.
In this section we discuss how to construct such a permutation isomorphism:
\begin{lemma}
\label{lem:pd}
    Let $G \leq \Sym \Omega$ be a primitive group of type PA.
    Then we can compute a non-abelian simple group $T \leq \Sym \Delta$,
    a positive integer $\ell$,
    and a permutation isomorphism from $G$ to a permutation group $\myhat
    G \leq \Sym(\Delta ^ \ell)$ such that the socle of $\myhat G$
    is $T ^ \ell$ in component-wise action on $\Delta ^ \ell$.
\end{lemma}
To this end we present the notion of \emph{permutation morphisms}
developed in \cite{siccha:phd-thesis}.
They arise from permutation isomorphisms by simply dropping the condition that
the domain map and the group homomorphism be bijections.
We illustrate how to use them to prove Lemma~\ref{lem:pd}.

\subsection{Basic Definitions}

For two maps $f : A \to B$ and $g : C \to D$ we denote by
$f \times g$ the product map
$A \times C \to B \times D, ~ (a,c) \mapsto (f(a), g(c))$.
For a right-action $\rho : \Omega \times G \to \Omega$ of a group $G$
and $g \in G$, $\omega \in \Omega$
we also denote $\rho(\omega, g)$ by $\omega ^ g$.

\begin{definition}\label{def:perm-hom}
    Let $G$ and $H$ be permutation groups on sets $\Omega$ and $\Delta$,
    respectively,
    let
    $f : \Omega \to \Delta$
    be a map,
    and let
    $\phi : G \to H$
    be a group homomorphism.
    Furthermore let $\rho$ and $\tau$ be the natural actions of $G$ and $H$ on
    $\Omega$ and $\Delta$, respectively.
    We call the pair $(f, \phi)$ a
    \emph{permutation morphism from $G$ to $H$ }
    if the following diagram commutes:
    \[
    \begin{tikzcd}[ampersand replacement=\&]
        \Omega \times G
            \ar[r, "\rho"]
            \ar[d, "f \times \phi"]
        \&
        \Omega
            \ar[d, "f"]
        \\
        \Delta \times H
            \ar[r, "\tau"]
        \&
        \Delta
    \end{tikzcd}
    \raisebox{-2em}{,}
    \]
    that is if $f(\omega ^ g) = f(\omega) ^ {\phi(g)}$
    holds for all $\omega \in \Omega$, $g \in G$.
    We call
    $\phi$ the \emph{group homomorphism of $(f, \phi)$}
    and $f$ the
    \emph{domain map of $(f, \phi)$}.
\end{definition}

It is immediate from the definition,
that the component-wise composition of two permutation morphisms again
yields a permutation morphism.
In particular we define the \emph{category of permutation groups}
as the category with all permutation groups as objects,
all permutation morphisms as morphisms,
and the component-wise composition as the composition of permutation
morphisms.
We rely on this categorical perspective in many of our proofs.

We denote a permutation morphism $F$ from a permutation group $G$ to a
permutation group $H$ by $F : G \to H$.
When encountering this notation keep in mind that $F$ itself is not a map but
a pair of a domain map and a group homomorphism.
We use capital letters for permutation morphisms.

It turns out that a
permutation morphism $F$
is a mono-, epi-, or isomorphism
in the categorical sense
if and only if both its domain map and group homomorphism are
injective, surjective, or bijective,
respectively.

For a permutation group $G \leq \Sym \Omega$ we call a map
$f : \Omega \to \Delta$ \emph{compatible with $G$} if there exists a group
homomorphism $\phi$ such that $F = (f, \phi)$ is a permutation morphism.
We say that a partition $\Sigma$ of $\Omega$ is $G$-invariant if
for all $A \in \Sigma$ and $g \in G$ we have $A ^ g \in \Sigma$.

\begin{lemma}[{\cite[Lemma~4.2.10]{siccha:phd-thesis}}]
\label{lem:perm-hom-iff-invariant-partition}
\label{lem:compatible-iff-invariant-partition}
    Let $G \leq \Sym \Omega$ be a permutation group
    and $f : \Omega \to \Delta$ a map.
    Then
    $f$ is compatible with $G$
    if and only if
    the partition
    of $\Omega$ into the non-empty fibers
    $
        \cset{ f ^ {-1}(\set{\delta}) }{ \delta \in f(\Omega) }
    $
    is $G$-invariant.
\end{lemma}

If $G$ is transitive, then the $G$-invariant partitions of $\Omega$ are
precisely the block systems of $G$.
Hence for a given blocksystem we can define a compatible map $f$ by sending
each point to the block it is contained in.

Let $G \leq \Sym \Omega$ be a permutation group
and $f : \Omega \to \Delta$ a surjective map compatible with $G$.
Then there exist a unique group $H \leq \Sym \Delta$
and a unique group homomorphism $\phi : G \to H$
such that $F := (f, \phi)$ is a permutation epimorphism,
see \cite[Corollary~4.2.7]{siccha:phd-thesis}.
We call $F$ the \emph{permutation epimorphism}
and $\phi$ the
\emph{group epimorphism of $G$ induced by $f$}.

\begin{example}
\label{exa:induced-perm-epi}
    Let $\Omega = \set{1, \ldots, 4}$,
    $a := (1,2)(3,4)$, $b := (1,3)(2,4)$,
    and
    $V := \gen{ a, b }$.
    Further consider the set
    $\Omega_1 := \set{1,2}$, the map
    $
        p_1 : \Omega \to \Omega_1,~
            1, 3 \mapsto 1,~
            2, 4 \mapsto 2,
    $
and the following geometric arrangement of the points $1, \ldots, 4$:
\vspace{-10pt}
\begin{figure}[H]
\centering
\begin{tikzpicture}
    \foreach \x in {1, ..., 2}
        \foreach \y in {1, ..., 2}
            \draw [fill] (\x, \y) circle [radius=0.050];
    \node [above left] at (1, 2) {{$1$}};
    \node [above left] at (2, 2) {{$2$}};
    \node [above left] at (1, 1) {{$3$}};
    \node [above left] at (2, 1) {{$4$}};
\end{tikzpicture}
\end{figure}
\vspace{-10pt}
    \noindent
    Observe that $a$ acts on $\Omega$ by permuting the points horizontally, while
    $b$ acts on $\Omega$ by permuting the points vertically.
    The map $p_1$ projects $\Omega$ vertically or ``to the top''.
    Notice how the fibers of $p_1$ correspond to a block-system of $V$.
    We determine the group epimorphism $\pi_1$ of $V$ induced by $p_1$.
    By definition $\pi_1(a)$ is the permutation which makes the following
    square commute:
    \[
    \begin{tikzcd}[ampersand replacement=\&]
        \Omega
            \ar[r, "a"]
            \ar[d, "p_1"]
        \&
        \Omega
            \ar[d, "p_1"]
        \\
        \Omega_1
            \ar[r, "\pi_1(a)"]
        \&
        \Omega_1
    \end{tikzcd}
    \]
    Take $1 \in \Omega_1$.
    We have $p_1 ^ {-1}(\set{1}) = \set{1,3}$,
    $a(\set{1,3}) = \set{2,4}$,
    and
    $p_1(\set{2,4}) = \set{2}$.
    Hence
    $\pi_1(a) = (1,2)$.
    Correspondingly we get $\pi_1(b) = \id_{\Omega_1}$.
\end{example}
\vspace{1em}

\subsection{Products of permutation morphisms}

For two permutation groups $H \leq \Sym \Delta$ and $K \leq \Sym \Gamma$
we define the \emph{product of the permutation groups $H$ and $K$}
as the permutation group
given by $H \times K$
in component-wise action
on $\Delta \times \Gamma$.
Correspondingly,
for an additional permutation group
$G \leq \Sym \Omega$
and two permutation morphisms $(f, \phi)$ and $(g, \psi)$ from
$G$ to $H$ and $K$, respectively,
we define the \emph{product permutation morphism $G \to H \times K$}
as $(f \times g, \phi \times \psi)$.
Iteratively, we define the product of several permutation groups or
permutation morphisms.

To prove Lemma~\ref{lem:pd} it suffices to be able
to compute the following:
given the socle $H \leq \Sym \Omega$ of a PA type group
compute a non-abelian simple group $T \leq \Sym \Delta$
and permutation epimorphisms, think projections,
$P_1, \ldots, P_\ell : H \to T$
such that the product morphism
$P : H \to T ^ \ell$ is an isomorphism.
Since every surjective map compatible with $H$ induces a unique permutation
epimorphism,
it in turn suffices to compute suitable maps $p_i : \Omega \to \Delta$.

\begin{example}
Consider the situation in Example~\ref{exa:induced-perm-epi}.
Let $P_1 := (p_1, \pi_1)$ and $\Omega_2 := \set{1,3}$.
Then the map $p_2 : \Omega \to \Omega_2,~ 1,2 \mapsto 1,~3,4 \mapsto 3$
is compatible with $V$ and induces the permutation epimorphism
$P_2 : V \to \gen{(1,3)}$.
The product maps $p_1 \times p_2 : \Omega \to \Omega_1 \times \Omega_2$
and $P_1 \times P_2 : V \to \gen{(1,2)} \times \gen{(1,3)}$
are isomorphisms of sets and permutation groups, respectively.
\end{example}
We illustrate how to construct one of the needed projections for PA type
groups.

\begin{example}
Let $\Delta = \set{1, \ldots, 5}$
and
$H := A_5 \times A_5 \leq \Sym(\Delta ^ 2)$.
We denote by $\mathbf 1_\Delta$ the trivial permutation group on $\Delta$.
The subgroup $H_1 := A_5 \times \mathbf 1_\Delta \leq \Sym(\Delta ^ 2)$
is normal in $H$.
Let us denote sets of the form
$\cset{ (\delta, x_2) }{ \delta \in \Delta }$ by $\set{(\ast, x_2)}$.
Then partitioning $\Delta ^ 2$ into orbits under $H_1$ yields the block system
\[
    \Sigma =
    \cset{ \set{ ( \ast, \delta_2 ) } }{ \delta_2 \in \Delta }.
\]

Note how mapping each $x \in \Delta ^ 2$ to the block of $\Sigma$ it is
contained in
is equivalent to
mapping each $x$ to $x_2$.
Thus we have essentially constructed the map
$p_2 : \Delta ^ 2 \to \Delta,~ x \mapsto x_2$.
Observe that we only used the group theoretic property that $H_1$ is a maximal
normal subgroup of $H$ and thus in particular did not use the actual product
structure of $\Delta ^ 2$.

Analogously we can construct the map
$p_1 : \Delta ^ 2 \to \Delta,~ x \mapsto x_1$.
For $i = 1, 2$ let $P_i : H \to A_5$ be the permutation epimorphisms of
$H$ induced by $p_1$ and $p_2$, respectively.
Since $p_1 \times p_2$ is an isomorphism,
$P_1 \times P_2$ must be a monomorphism.
By order arguments $P_1 \times P_2$ is thus an isomorphism.
\end{example}

In general the above construction does not yield permutation epimorphisms with
identical images.
We can alleviate this by computing elements of the given
group which conjugate the minimal normal subgroups of its socle to each other.
For the general construction see the
\emph{(homogenized) product decomposition by minimal normal subgroups}
in \cite[Definitions~5.1.3 and~5.1.5]{siccha:phd-thesis}.
Lemma~\ref{lem:pd} then follows from
\cite[Corollary~5.19]{siccha:phd-thesis}.

\section{Reduction Homomorphism}
\label{sec:red}

Recall from Section~\ref{sec:strat} that a key ingredient of our second phase
is a group homomorphism which reduces the original problem on $n$ points to a
problem on less or equal than $6 \log n$ points.
We illustrate shortly how to construct this homomorphism,
for the details refer to \cite[Theorem~9.1.6]{siccha:phd-thesis}.

Let $G \leq \Sym \Omega$ be a primitive group with non-regular socle
and $T \leq \Sym \Delta$ be a
\emph{socle-component of $G$}, confer
\cite[Chapters~5~and~7]{siccha:phd-thesis} for a definition.
Then $T$ is a non-abelian
simple group,
there exists a positive integer $\ell$
such that $\Soc G$ is isomorphic to $T ^ \ell$,
and
by
\cite[Lemma~2.6.1]{siccha:phd-thesis}
we have
$\abs \Omega = \abs \Delta ^ s$ for some $s \in \set{\ell / 2, \ldots, \ell}$.
Denote by $R$ the permutation group induced by the right-regular action of
$\Out T$ on itself.
We show that we can evaluate the following two group
homomorphisms:
first an embedding
$N_{\Sym \Omega}(\Soc G) \to \Aut T \wr S_\ell$
and second an epimorphism
$\Aut T \wr S_\ell \to R \wr S_\ell$,
where $R \wr S_\ell$ is the imprimitive wreath product
and thus acts on $\abs R \cdot \ell$ points.

We sketch the proof that $\abs R \cdot \ell \leq 6 \log n$.
Let $m := \abs \Delta$ and $r := \abs R$.
Note that for $\ell$ we have $\ell \leq 2 s = 2 \log_m n$.
Since $R$ is regular, we have $r = \abs {\Out T}$.
Since $T$ is a socle-component of $G$, we have $\abs{\Out T} \leq 3 \log m$
by \cite[Lemma~7.7]{guralnick-maroti-pyber:normalizers-primitive-groups}.
In total we have
$r \cdot \ell \leq 3 \log m \cdot 2 \log_m n = 6 \log n$.

In our implementation we use a modified version of this reduction.
For groups of type PA we can directly compute an isomorphism from the product
action wreath product into the corresponding imprimitive wreath product.

\section{Implementation}
\label{sec:implementation}
A version of our normalizer algorithm
for groups of type PA is implemented in the GAP package
\texttt{NormalizersOfPrimitiveGroups}.

Table~\ref{tab:runtimes} shows a comparison of runtimes of our algorithm
and the GAP function \texttt{Normalizer}.
At the time of writing, there are two big bottlenecks in the implementation.
First, the GAP built-in algorithm to compute the socle of a group is
unnecessarily slow.
State-of-the-art algorithms as in~%
\cite{cannon-holt:chief-series-composition-series-socle}
are not yet implemented.
Secondly, computing a permutation which transforms a given product
decomposition into a so-called natural product decomposition currently also is
slow.
The latter may be alleviated by implementing the corresponding routines in
for example C \cite{c-lang} or Julia \cite{julia-lang}.
Note that the actual normalizer computation inside the normalizer of the socle
appears to be no bottleneck:
in the example with socle type $(A_5) ^ 7$ it took only 40ms!

\begin{table}
\centering
\caption{Table with runtime comparison.}\label{tab:runtimes}
\begin{tabular}{rrrr}
~~Socle type & ~~~~Degree & ~~Our algorithm & ~~GAP built-in alg. \\
\hline
\\
$(A_5)^ 2$         &   25      &     24ms   &    200ms     \\
$(A_5)^ 3$         &   125     &     50ms   &   1500ms     \\
$(A_5)^ 4$         &   625     &    300ms   &  29400ms       \\
$(A_5)^ 7$         &   78125   &    67248ms &       --       \\
\\

$PSL(2, 5) ^ 2$   &   36      &     40ms   &     300ms     \\
$PSL(2, 5) ^ 3$   &   216     &     90ms   &    1900ms     \\
$PSL(2, 5) ^ 4$   &   1296    &    400ms   &   64000ms     \\
\\

$(A_7)^ 2$         &   49      &     38ms   &    900ms     \\
$(A_7)^ 3$         &   343     &    200ms   &  16800ms       \\
$(A_7)^ 4$         &   2401    &   1400ms   & 839000ms     \\
\end{tabular}
\end{table}

\section{Acknowledgments}

Substantial parts of the work presented in this article
were written while the author was supported
by the
German Research Foundation (DFG) research
training group ``Experimental and constructive algebra''
(GRK 1632)
and employed
by the
Lehr- und Forschungsgebiet Algebra, RWTH Aachen University
and  the
Department of Mathematics, University of Siegen.

\printbibliography

@misc{recog1.3.2,
  author = {Neunh{\"o}ffer,    M.    and   Seress,   {\a'A}.   and
  Ankaralioglu, N. and Brooksbank, P. and Celler, F. and Howe,  S. and Law, M.
  and Linton, S. and Malle, G. and Niemeyer,  A.  and O'Brien, E. and
  Roney-Dougal, C. M.  and Horn, M.},
  title = {{recog}, A collection  of group recognition methods, {V}ersion 1.3.2},
  month = {Apr},
  year =  {2018},
  note =  {~GAP package},
  url =   {https://gap-packages.github.io/recog},
}

@article{cannon-holt:chief-series-composition-series-socle,
    title = "Computing Chief Series, Composition Series and Socles in Large
            Permutation Groups",
    journal = "Journal of Symbolic Computation",
    volume = "24",
    number = "3",
    pages = "285 - 301",
    year = "1997",
    issn = "0747-7171",
    doi = "10.1006/jsco.1997.0127",
    author = "John Cannon and Derek Holt",
}

@article{sims:backtrack,
  title = {Determining the conjugacy classes of permutation groups},
  author = {Sims, Charles C.},
  year = {1971},
  journal = {Computers in Algebra and Number Theory},
  volume = {4},
  pages = {191--195},
  publisher = {American Mathematical Society},
}

@article{butler:backtrack,
  title = {Computing in permutation and matrix groups. II. Backtrack algorithm},
  author = {Butler, Gregory},
  journal = {Mathematics of Computation},
  volume = {39},
  year = {1982},
  pages = {671--680},
  doi = {10.1090/S0025-5718-1982-0669659-5},
  issn = {1088-6842},
}

@article{hulpke:normalizers,
  title={Normalizer calculation using automorphisms},
  author={Hulpke, Alexander},
  journal={Computational Group Theory and the Theory of Groups. AMS special
      session On Computational Group Theory, Davidson, USA},
  pages={105--114},
  year={2007},
  doi = {10.1090/conm/470}
}

@article{glasby-slattery:normalizers,
title = "Computing intersections and normalizersin soluble groups",
journal = "Journal of Symbolic Computation",
volume = "9",
number = "5",
pages = "637 - 651",
year = "1990",
issn = "0747-7171",
doi = "https://doi.org/10.1016/S0747-7171(08)80079-X",
author = "Stephen P. Glasby and Michael C. Slattery",
}

@article{butler:normalizers,
  title={Computing normalizers in permutation groups},
  author={Butler, Gregory},
  journal={Journal of algorithms},
  volume={4},
  number={2},
  pages={163--175},
  year={1983},
  publisher={Elsevier},
  doi={10.1016/0196-6774(83)90043-3}
}

@article{holt:normalizers,
title = "The computation of normalizers in permutation groups",
journal = "Journal of Symbolic Computation",
volume = "12",
number = "4",
pages = "499 - 516",
year = "1991",
issn = "0747-7171",
doi = "10.1016/S0747-7171(08)80100-9",
author = "D.F. Holt",
}

@book{c-lang,
  title={The C programming language},
  author={Kernighan, Brian W and Ritchie, Dennis M},
  year={1988},
  isbn={978-0131103627},
  publisher={Prentice Hall},
  edition={Second},
}

@article{julia-lang,
  title={Julia: A fresh approach to numerical computing},
  author={Bezanson, Jeff and Edelman, Alan and Karpinski, Stefan and Shah, Viral B},
  journal={SIAM review},
  volume={59},
  number={1},
  pages={65--98},
  year={2017},
  publisher={SIAM},
  url={https://doi.org/10.1137/141000671}
}

@article{leon,
    author = {Leon, Jeffery S.},
    title = {Permutation group algorithms based on partitions {I} {T}heory and
    algorithms},
    journal = {Journal of Symbolic Computation},
    year = {1991},
    pages = {533–583},
    volume = {12},
}

@article{JPWW,
  title={Permutation group algorithms based on directed graphs},
  author={Jefferson, Christopher and Pfeiffer, Markus and Waldecker, Rebecca and Wilson, Wilf A},
  journal={arXiv preprint arXiv:1911.04783},
  year={2019}
}

@manual{gap,
    key          = {GAP},
    organization = {The GAP-Group},
    title        = {GAP -- Groups, Algorithms, and Programming,
                    Version 4.11.0},
    year         = 2020,
    url          = {https://www.gap-system.org},
}

@article {magma,
    AUTHOR = {Bosma, Wieb and Cannon, John and Playoust, Catherine},
     TITLE = {The {M}agma algebra system. {I}. {T}he user language},
      NOTE = {Computational algebra and number theory (London, 1993)},
   JOURNAL = {J. Symbolic Comput.},
  FJOURNAL = {Journal of Symbolic Computation},
    VOLUME = {24},
      YEAR = {1997},
    NUMBER = {3-4},
     PAGES = {235--265},
      ISSN = {0747-7171},
   MRCLASS = {68Q40},
  MRNUMBER = {MR1484478},
       DOI = {10.1006/jsco.1996.0125},
}

@thesis{siccha:phd-thesis,
    author = {Siccha, Sergio},
    title = {Normalizers of primitive groups with non-regular socles in
    polynomial time},
    type = {Ph.D. Thesis},
    school = {RWTH Aachen University, To appear},
}

@article{mckay-piperno:nauty-traces,
    title = {Practical graph isomorphism, II},
    journal = "Journal of Symbolic Computation ",
    volume = "60",
    number = "0",
    pages = "94 - 112",
    year = "2014",
    issn = "0747-7171",
    doi = "http://dx.doi.org/10.1016/j.jsc.2013.09.003",
    author = "Brendan D. McKay and Adolfo Piperno"
}

@article{guralnick-maroti-pyber:normalizers-primitive-groups,
    title = "Normalizers of primitive permutation groups",
    author = "Guralnick, Robert M and Mar{\'o}ti, Attila and Pyber, L{\'a}szl{\'o}",
    journal = "Advances in Mathematics",
    volume = "310",
    pages = "1017--1063",
    year = "2017",
    publisher = "Elsevier",
    doi = "10.1016/j.aim.2017.02.012"
}

@article{liebeck-praeger-saxl:onan-scott,
    title={On the {O}'{N}an-Scott theorem for finite primitive permutation
    groups},
    author={Liebeck, Martin W. and Praeger, Cheryl E. and Saxl, Jan},
    volume={44},
    DOI={10.1017/S144678870003216X},
    number={3},
    journal={Journal of the Australian Mathematical Society. Series A. Pure Mathematics and Statistics},
    publisher={Cambridge University Press},
    year={1988},
    pages={389--396},
}

@article{praeger:onan-scott-eight-types,
    title={The inclusion problem for finite primitive permutation groups},
    author={Praeger, Cheryl E},
    journal={Proceedings of the London Mathematical Society},
    volume={3},
    number={1},
    pages={68--88},
    year={1990},
    publisher={Wiley Online Library},
    doi = {10.1112/plms/s3-60.1.68},
}

@article{babai:graphiso-quasipolynomial,
    author = {{Babai}, L{\'a}szl{\'o}},
    title = "{Graph Isomorphism in Quasipolynomial Time}",
    journal = {arXiv e-prints},
    year = "2015",
    month = "Dec",
    eid = {arXiv:1512.03547},
    archivePrefix = {arXiv},
    eprint = {1512.03547},
    primaryClass = {cs.DS},
}

@inproceedings{babai:graphiso-quasipolynomial-acm,
 author = {Babai, L\'{a}szl\'{o}},
 title = {Graph Isomorphism in Quasipolynomial Time [Extended Abstract]},
 booktitle = {Proceedings of the Forty-eighth Annual ACM Symposium on Theory of Computing},
 series = {STOC '16},
 year = {2016},
 isbn = {978-1-4503-4132-5},
 location = {Cambridge, MA, USA},
 pages = {684--697},
 numpages = {14},
 doi = {10.1145/2897518.2897542},
 publisher = {ACM},
}

@article{luks-miyazaki:normalisers-poly,
    author = "Luks, Eugene and Miyazaki, Takunari",
    title = "Polynomial-time normalizers",
    year = "2011",
    volume = "13",
    number = "4",
    pages= "61--96",
    journal = "Discrete Mathematics and Theoretical Computer Science",
    eprint = "hal-00990473",
    eprinttype = "hal"
}

@article{schweitzer-wiebking:canonisation-framework,
    author = {Schweitzer, Pascal and Wiebking, Daniel},
    title = {A unifying method for the design of algorithms canonizing
            combinatorial objects},
    journal = {arXiv e-prints},
    year = "2018",
    month = "Jun",
    eid = {arXiv:1806.07466},
    archivePrefix = {arXiv},
    eprint = {1806.07466},
    primaryClass = {cs.DS},
}

@article{wiebking:normalizers-simply-exponential,
    author = {Wiebking, Daniel},
    title = {Normalizers and permutational isomorphisms in simply-exponential
            time},
    journal = {arXiv e-prints},
    year = "2019",
    month = "Apr",
    eid = {arXiv:1904.10454},
    archivePrefix = {arXiv},
    eprint = {1904.10454},
    primaryClass = {cs.DS},
}

@inproceedings{wiebking:normalizers-simply-exponential-siam,
  author    = {Daniel Wiebking},
  title     = {Normalizers and permutational isomorphisms in simply-exponential time},
  booktitle = {Proceedings of the Thirty-First Annual {ACM-SIAM} Symposium on Discrete
               Algorithms, {SODA} 2020, Salt Lake City, Utah, USA, January 5-8,
               2020},
  year      = {2020},
}

@book{dixon-mortimer:permutation-groups,
    author={Dixon, John D and Mortimer, Brian},
    title={Permutation Groups},
    volume={163},
    year={1996},
    publisher={Springer Science \& Business Media},
    ISBN = {978-0-387-94599-6},
    doi = "10.1007/978-1-4612-0731-3",
}

@thesis{theissen,
    title={Eine Methode zur Normalisatorberechnung in Permutationsgruppen mit
    Anwendungen in der Konstruktion primitiver Gruppen},
    year = 1997,
    author = {Heiko Theißen},
    type = {Ph.D. Thesis},
    school = {RWTH Aachen University},
}

\end{document}